\documentclass[10pt]{article}
\usepackage{amsmath,amsfonts,latexsym,amssymb,amsthm, verbatim}
\usepackage{url}

\newcommand{\Q}{\mathbb{Q}}

\newcommand{\Z}{\mathbb{Z}}

\DeclareMathOperator{\Gal}{Gal}
\DeclareMathOperator{\GL}{GL}

\begin{document}

\title{On the number of elliptic curves with prescribed isogeny or torsion group over number fields of prime degree}
\author{Filip Najman}
\date{}
\maketitle
\begin{abstract}
Let $p$ be a prime and $K$ a number field of degree $p$. We determine the finiteness of the number of elliptic curves, up to $\overline K$-isomorphism, having a prescribed property, where this property is either that the curve contains a fixed torsion group as a subgroup, or that it has a cyclic isogeny of prescribed degree.
\end{abstract}

\section{Introduction}

Our goal is to study the number of elliptic curves over number fields of prime degree $K/\Q$, up to $\overline K$-isomorphism, having a prescribed property. This property will be one of the following:
\begin{itemize}
\item[1)] The curve's group of $K$-rational points contains a subgroup isomorphic to a fixed torsion group $T$.

\item[2)] The curve admits a cyclic $n$-isogeny defined over $K$, for some prescribed integer $n$.

\end{itemize}

Our methods will actually solve the first problem for number fields of degree $d$, where $d$ is not divisible by any of 2, 3 or 5.

It is clear that, when counting the number of elliptic curves over $K$ with a cyclic $n$-isogeny, one has to count up to $\overline K$-isomorphism, since if $E/K$ has a cyclic $n$-isogeny, then so does every quadratic twist of $E$ (in other words, the modular curve $Y_0(n)$ classifying cyclic $n$-isogenies is a coarse moduli space).

On the other hand, it is not clear whether it is better to count elliptic curves with prescribed torsion up to $K$-isomorphism or $\overline K$-isomorphism. Since we will mostly be concerned with the finiteness of the number of curves with given torsion, and since (see for example \cite[Lemma 5.5]{mr}) the number of elliptic curves over $K$ whose group of $K$-rational points contains a subgroup isomorphic to $T$ is finite up to $\overline K$-isomorphism if and only if it is finite up to $K$-isomorphism, we see that the choice of the isomorphism does not really matter.

Thus, for simplicity's sake, we will always count elliptic curves up to $\overline K$-isomorphism, and we will do so, usually without mention, throughout the paper.

 The number of curves will naturally depend on the prescribed property. Let $m|n$. Folowing \cite{km}, we define
 $$\Gamma_1(m,n)=\left\{\begin{pmatrix} a&b\\ c&d\end{pmatrix} \in \GL_2(\Z)| a-1\equiv c\equiv 0\text{ mod } n,\ b\equiv d-1\equiv 0\text{ mod } m\right\},$$
 and let $X_1(m,n)$ be $\Gamma_1(m,n)\backslash \mathcal H^*$. The curve $X_1(m,n)$ is defined over $\Q(\zeta_m)$. Let $Y_1(m,n)=X_1(m,n)\backslash \{cusps\}$. The moduli interpretation of $Y_1(m,n)$ is that its $K$-rational points classify isomorphism classes of the triples $(E, P_m, P_n)$, where $E$ is an elliptic curve (over $K$) and $P_m$ and $P_n$ are torsion points (over $K$) which generate a subgroup isomorphic to $\Z/m \Z \oplus \Z/n \Z$. Note that $X_1(1,n)=X_1(n)$ and $Y_1(n)=Y_1(1,n)$. Similarly $X_0(n)=\Gamma_0(n) \backslash \mathcal H^*$ and $Y_0(n)=X_0(n)$ $\backslash \{cusps\}$. The moduli interpretation of $Y_0(n)$ is that its $K$-rational points classify isomorphism classes of pairs $(E,C)$, where $E/K$ is an elliptic curve and $C$ is a cyclic (Gal$(\overline K/K)$-invariant) subgroup of $E$.

 Let $Y$ be the modular curve ($Y_1(m,n)$ or $Y_0(n)$) corresponding to the property we are studying, and $X$ its compactification. It is clear that the number of elliptic curves with the desired property over $K$ is positive and finite if and only if $0<|Y(K)|<+\infty$. Hence, we will from now on, express our problem in determining when $0<|Y(K)|<+\infty$, or equivalently $|\{K\text{-rational cusps of }X\}|<|X(K)|<+\infty$.

The key value that will determine whether $0<|Y(K)|<+\infty$ is its genus. If $Y$ is a genus 0 curve, then $Y(K)$ has either none or infinitely many points, thus implying that there are none or infinitely many curves with the given property. 

In particular, for $X=X_1(m,n)$ of genus 0 and $m>2$, that is for $(m,n)=(3,3),$ $(3,6),$ $(4,4),$ or $(5,5)$, $|Y(K)|=\infty$ if $K$ contains $\zeta_m$, an $m$-th primitive root of unity, and zero otherwise. One direction follows from the Weil pairing, and the other from the fact that the curve $X_1(m,n)$ has a $\Q(\zeta_m)$-rational point (and hence infinitely many) for all of the listed pairs.

 On the other hand, if $X$ is of genus $\geq 2$, then by Faltings' theorem \cite{fal}, $X(K)$ has only finitely many points. Thus, we see that we are left only with the genus 1 case.

 In Section 2 we study $|Y_1(m,n)(K)|$; combining the obtained results with the results of \cite{kn} and \cite{naj}, we solve this problem for number fields of prime degree completely.



 In Section 3 we study the number of points on $Y_0(n)$, for a prescribed $n$, over prime degree fields $K$.


 \section{Number of elliptic curves with prescribed torsion}
 \label{sec:tor}
  As stated in the introduction, we need to study the modular curves $X_1(m,n)$ of genus 1. One can see in \cite[Theorem 2.6, Proposition 2.7]{jkp} that $X_1(m,n)$ is of genus 1 only for $(m,n)=(1,11),\ (1,14),\ (1,15),\ (2,10),\ (2,12),$ $(3,9),\ (4,8)$ or $(6,6)$. Note that if $K$ is a field of odd degree, $m>2$ and $m|n$, then $|Y_1(m,n)(K)|=0$, as otherwise $\Q (\zeta_m )$ would have to be a subfield of $K$ (because of the Weil pairing \cite[Corollary 3.11]{was}), which is impossible. Thus we can ignore the cases $(3,9),\ (4,8)$ and $(6,6)$, implying that we are left with the cases
 \begin{equation}
 (m,n)\in \{ (1,11),\ (1,14),\ (1,15),\ (2,10),\ (2,12)\}.
 \label{popis}
  \end{equation}

  The number of points on $Y_1(m,n)$ was determined by Kamienny and the author \cite{kn} over all quadratic fields and by the author \cite{naj} over all cubic fields. In both cases there were exceptional cases where, there are a few fields $K$ over which $0<|Y_1(m,n)(K)|<+\infty$.  Note also that in \cite{naj2} it is proven that there are infinitely many such quartic fields. In this paper, we prove that there are no such number fields of degree $d$, where $d$ is not divisible by any of $2, 3$ or 5. For $d=5$, we find that $|Y_1(11)(\Q(\zeta_{11})^+|=15$ (where $\Q(\zeta_{11})^+$ is the maximal real subfield of $\Q(\zeta_{11})$) and $|Y_1(m,n)(K)|=0$ or $\infty$ for all other triples $(K,m,n)$, where $K$ is a quintic field and $(m,n)$ are listed in \eqref{popis}.

 We state our results in the theorem.

 \newtheorem{tm}{Theorem}
 \begin{tm}
 \label{t1} \newtheorem{cor}[tm]{Corolaary}
 \begin{itemize}
 \item[a)] Let $d$ be a positive integer coprime to $6$ and $K$ a number field of degree $d$. Then
 \begin{equation}
 \label{eq1}
 |Y_1(m,n)(K)|=\begin{cases} 0 \text{ or }\infty & \text{if } (m,n)\in \{(1,14), (1,15), (2,10), (2,12)\} \\
                            0  & \text{if } (m,n)\in \{(3,9), (4,8), (6,6)\} \end{cases}
\end{equation}

\item[b)] Let $K$ a number field of degree $d$. Then
\begin{equation}
\label{eq11}
|Y_1(11)(K)|= \begin{cases} 0 \text{ or } \infty & \text{if } 3,4,5\nmid d \text{ or if } d=5 \text{ and }K\neq \Q(\zeta_{11})^+\\
                            15 & \text{if } K=\Q(\zeta_{11})^+\end{cases}
\end{equation}

\end{itemize}
\end{tm}

\textbf{Remark 1}.

An elliptic curve over $\Q$ has one of the following 15 torsion groups (see \cite{maz} for the proof):
$$\Z / n \Z, \text{ where }n=1\ldots 10, 12$$
\begin{equation}
\Z /2\Z \oplus \Z / 2n\Z \text{ where }n=1\ldots 4,
\label{eq2}
\end{equation}
For any number field $K$, there are infinitely many elliptic curves containing any of the groups from (\ref{eq2}). All of the groups (\ref{eq2}) are parameterized by modular curves of genus 0. On the other hand, the groups from (\ref{eq1}) and (\ref{eq11}) are parameterized by curves of genus 1.

We can summarize our results in terms of elliptic curves with prescribed torsion groups. The following corollary follows trivially from Theorem \ref{t1} and Remark 1.

\newtheorem{cor3}[tm]{Corollary}
\begin{cor3}
\label{cr}
 Let $d$ be a positive integer not divisible by any of $2$, $3$ or $5$, and let be $K$ a number field of degree $d$. If $T$ is one of the following groups
$$\Z / n \Z, \text{ where }n=1,\ldots , 12, 14, 15,$$
\begin{equation}
\Z /2\Z \oplus \Z / 2n\Z \text{ where }n=1,\ldots ,6,
\label{ej3}
\end{equation}
then there are either none or infinitely many elliptic curves over $K$ containing $T$ as a subgroup. If $T$ is any other finite group, then there are only finitely many (maybe $0$) elliptic curves containing $T$ as a subgroup.
\end{cor3}

\medskip

 Let $Y_1(m,n)$ be of genus 1. For $0<|Y_1(m,n)(K)|<+\infty$ to be true, the following two conditions have to hold.
\begin{itemize}

\item[(C1)] $rank(X_1(m,n)(K))=0$.

\item[(C2)] At least one of the torsion points of $X_1(m,n)(K)$ is not a cusp (note that this implies $|X_1(m,n)(K)|_{tors}>|X_1(m,n)(\Q)|_{tors}$).
\end{itemize}

To prove Theorem \ref{t1}, we will use division polynomials and Galois representations.

We denote by $\psi_n$ the $n$-th division polynomial, which has the property that, for a point $P\neq \mathcal O$ on an elliptic curve in Weierstrass form, $\psi_n(x(P))=0$ if and only if $nP=0$. For a definition and more information on division polynomials see \cite{was}. Note that the field of definition of $P$ will be either the number field $F$ obtained by adjoining a root of $\psi_n$ or a quadratic extension of $F$.

Let $E[n]$ denote the $n$-torsion subgroup of $E/\Q$ over $\overline \Q$ and let $\Q(E[n])$ be the $n$-th division field of $E$. The Galois group $\Gal(\overline \Q /\Q)$ acts on $E[n]$ and gives rise to an embedding $\rho_n: \Gal(\Q(E[n]/\Q)  \hookrightarrow \GL_2(\Z/n\Z)$ called a \emph{Galois representation associated to $E[n]$}. Serre's Open Image Theorem \cite{ser} tells us that for an elliptic curve without complex multiplication (all of the $X_1(m,n)$-s that we study do not have complex multiplication) this embedding is surjective for all but finitely many primes. The primes for which this embedding is possibly not surjective can be computed in SAGE \cite{sag} with the {\sf non\_surjective()} function, which uses the bounds for the non-surjective primes from \cite{co}.

To prove Theorem \ref{t1}, we will find the fields $K$ of prime degree over which $X_1(m,n)(K)_{tors}$ is strictly larger than $X_1(m,n)(\Q)_{tors}$, where $(m,n)$ is a pair from the list (\ref{popis}).

\begin{proof}[Proof of Theorem \ref{t1}]

Before proceeding with the details of the proof, we sketch the strategy of the proof.

Let $E/ \Q$ be an elliptic curve without complex multiplication. Let $l$ be an odd prime such that the Galois representation $\rho_l$ is surjective. This implies that the division polynomial $\psi_l$ is irreducible. The field of definition $F$ of a torsion point of order $l$ has to contain the field obtained by adjoining a root of $\psi_l$ to $\Q$, which is a field of degree $(l^2-1 )/2$. In particular, the degree of $F$ must be divisible by 4. It follows that if $K$ is any number field of prime degree then $E(K)$ cannot contain a point whose order is divisible by any odd prime $l$ for which $\rho_l$ is surjective.

If $E(\Q)$ has at least one point of order 2 and $K$ is a number field of odd degree, then by \cite[Lemma 1]{naj}, the 2-Sylow subgroup of $E(K)$ is equal to the 2-Sylow subgroup of $E(\Q)$. Suppose now $E(\Q)[2]=0$. One can then write $E(\Q)$ in short Weierstrass form $y^2=f(x)$, where $f(x)$ is an irreducible degree 3 polynomial, and conclude that any number field containing a point of order 2 has to have the field obtained by adjoining a root of $f$ to $\Q$ as a subfield. Thus $E(\Q)[2]=0$ implies $E(K)[2]=0$ when the degree of $K$ is not divisible by $3$.

Combining these two cases, we see that for any number field $K$ of degree coprime to 6, the 2-Sylow subgroup of $E(K)$ is equal to the 2-Sylow subgroup of $E(\Q)$.

As the $X_1(m,n)$ we study do not have complex multiplication, it follows that condition (C2) can be true only if there exists a prime $l$ such that $\rho_l$ is non surjective and a prime degree field $K$ contains a root of $\psi_l$.

For all the $X_1(m,n)$-s we study, we list the primes (computed in SAGE) for which the Galois representation $\rho_l$ is not surjective in the table below.
\begin{center}
\begin{tabular}{|c|c|c|}
\hline
\text{Modular curve} & \text{Cremona Label \cite{cre}} &  \text{primes $l$ for which $\rho_l$ is not surjective}\\
\hline
$X_1(11)$ & 11a3 & 5\\
\hline
$X_1(14)$ & 14a4 & 2, 3\\
\hline
$X_1(15)$ & 15a8 & 2\\
\hline
$X_1(2,10)$ & 20a2  & 2, 3\\
\hline
$X_1(2,12)$ & 24a4 & 2 \\
\hline
\end{tabular}

\medskip

Table 1.
\end{center}

We start with $X_1(11)$ and the prime $5$. The curve elliptic curve $X_1(11)$ has the following Weierstrass model
$$y^2-y=x^3-x^2$$
and $X_1(11)(\Q)\simeq \Z/5\Z$ and all the rational torsion points are cusps. Since no number field of odd degree contains $\Q(\zeta_5)$ as a subfield, $X_1(11)(F)[5]\simeq \Z/5\Z$, for all number fields $F$ of odd degree. We now factor the 25-division polynomial of $X_1(11)$ and obtain the factorization $$\psi_{25}=\psi_5f_5g_5f_{20}g_{20}f_{250},$$
where $f_n$ and $g_n$ are irreducible polynomials of degree $n$. We need to examine only the fields generated by $f_5$ and $g_5$, as the zeroes of $\psi_5$ correspond to the points of order 5, which are either rational (and thus cusps) or are defined over a field of even degree, and as $f_{20},\ g_{20}$ and $f_{250}$ generate fields of even degree.

We compute that $f_5$ and $g_5$ generate the same field, $\Q(\zeta_{11})^+$. We compute $X_1(11)(\Q(\zeta_{11})^+)_{tors}\simeq \Z /25 \Z$ and $rank(X_1(11)(\Q(\zeta_{11})^+))=0$ (via 2-descent). The curve $X_1(11)$ has 10 cusps over $\Q(\zeta_{11})^+$ (see \cite[Example 9.3.5]{dim}), thus there are 15 non-cuspidal points on $X_1(11)(\Q(\zeta_{11})^+)$. As $\Z / 11 \Z$ has 10 generators and each point on $X_1(11)(\Q(\zeta_{11})^+)$ corresponds to a pair $(E,\pm P)$, where $P$ is a point of order 11, we see that this implies that there are $15/5=3$ elliptic curves with a point of order 11 over $\Q(\zeta_{11})^+$. We also conclude that $\Q(\zeta_{11})^+$ is the only quintic field $K$ such that there is a positive finite number of elliptic curves with a $K$-rational point of order 11 and that if $K$ is a number field of degree not divisible by 2, 3 or 5 (as to not contain $\Q(\zeta_{11})^+$) then there are either none or infinitely many elliptic curves with a $K$-rational point of order 11.

We move on to $X_1(14)$ and the prime 3. The elliptic curve $X_1(14)$ has a Weierstrass model
$$y^2 + xy + y = x^3 - x$$
and $X_1(14)(\Q)\simeq \Z/6\Z$, where all the rational torsion points are cusps.
We factor its 9-division polynomial $\psi_9$ and obtain
$$\psi_9=\psi_3 f_3f_6f_{27},$$
where $f_n$ is an irreducible polynomial of degree $n$. As in the previous case, the zeroes of $\psi_3$ do not concern us. From the factorization of $\psi_9$, we see that there are no points on $X_1(14)$ of order 9 over any number field of degree not divisible by 3. Thus $X_1(14)$ has no non-rational torsion points over any number field of degree not divisible by 2 or 3. Note that points on $X_1(m,n)$ over cubic fields have already been dealt with in \cite{naj}.

The case $X_1(2,10)$ and the prime 3 remain. The elliptic curve $X_1(2,10)$ has a Weierstrass model
$$y^2 = x^3 + x^2 - x$$
and $X_1(2,10)(\Q)\simeq \Z/6\Z$, where all the rational torsion points are cusps. We factor its 9-division polynomial $\psi_9$ and obtain
$$\psi_9=\psi_3f_9f_{27},$$
where $f_n$ is an irreducible polynomial of degree $n$. Similarly as in the previous case we can conclude that there are no points on $X_1(2,10)$ of order 9 over any number field of degree not divisible by 9. Thus $X_1(2,10)$ has no non-rational torsion points over any number field of degree not divisible by 2 or 9.

Combining all the cases we have proven Theorem \ref{t1}. \end{proof}

\textbf{Remark 2.} Let $\alpha$ be the root of the polynomial $x^5 - 18x^4 + 35x^3 - 16x^2 - 2x + 1$ which generates $\Q(\zeta_{11})^+$. The 3 elliptic curves over $\Q(\zeta_{11})^+$ with a point of order 11 are

$$y^2 + \frac{4\alpha^4 - 73\alpha^3 + 150\alpha^2 - 96\alpha +
    27}{11}xy + \frac{3\alpha^4 - 19\alpha^3 + 19\alpha^2 - 6\alpha + 1}{11}y =$$
\begin{equation}
\label{kr1}
    = x^3 +
\frac{3\alpha^4 - 19\alpha^3 + 19\alpha^2 - 6\alpha + 1}{11}x^2,
\end{equation}
$$
y^2 + \frac{2\alpha^4 - 32\alpha^3 + 68\alpha - 8}{11}xy +
\frac{-6\alpha^4 + 111\alpha^3 - 261\alpha^2 + 190\alpha - 33}{11}y =$$

\begin{equation}= x^3 + \frac{-6\alpha^4
\label{kr2}
    + 111\alpha^3 - 261\alpha^2 + 190\alpha - 33}{11}x^2
\end{equation}
and

 $$y^2 + \frac{xy}{11}(45\alpha^4 - 799\alpha^3 + 1379\alpha^2 -
    372\alpha - 179)+ \frac{y}{11}(372\alpha^4 - 6605\alpha^3 + $$
\begin{equation} 11404\alpha^2 - 3157\alpha -1519) = x^3 +
\frac{x^2}{11}(372\alpha^4 - 6605\alpha^3 + 11404\alpha^2 - 3157\alpha -
    1519).
    \label{kr3}
\end{equation}
The curves (\ref{kr1}), (\ref{kr2}), (\ref{kr3}) have $j$-invariants $-11\cdot131^3$, $-2^{15}$ and $-11^2$ respectively. The curve (\ref{kr2}) has complex multiplication by the ring of integers $\Q(\sqrt{-11})$, while the other two curves do not have complex multiplication. The elliptic curves above were constructed using the formulas from \cite{sut}.

\medskip

\section{Number of elliptic curves with a cyclic $n$-isogeny}

As we are studying elliptic curves with a cyclic $n$-isogeny, we are led to the study of the modular curves $X_0(n)$.

If $X_0(n)$ is of genus 0 (when $n\leq 10$, $n=12,$ $13,$ $16,$ $18$ or 25), then there are already infinitely many elliptic curves with a cyclic $n$-isogeny over $\Q$.

Let $S$ be the set of $n$-s such that $X_0(n)$ is of genus 1, i.e. \begin{equation}
\label{s}
S=\{11,14,15,17,19,20,21,24,27,32,36,49\}.
\end{equation}
Note that the Cremona label of $X_0(n)$ is $na1$ and, as listed in Cremona's tables \cite{cre}, the curves $X_0(n)$ have rank 0 over $\Q$ and hence $|Y_0(n)(\Q)|$ is finite for all $n\in S$. The curve $Y_0(n)(\Q)$ has 1 point for $n=19$ and 27, 2 points for $n=14$ and $17$, 3 points for $n=11$, 4 points for $n=15$ and 21 and 0 points for the remaining genus 1 cases \cite{maz2}.

As in the previous section, we will be interested in finding the fields $K$ of prime degree for which the elliptic curve $X_0(n)$ has larger torsion than over $\Q$, but still has rank 0.

We state our results in the following theorem.
\newtheorem{tm3}[tm]{Theorem}
\begin{tm3}
\label{t3}
The only pairs $(n,K)$, where $n\in S$ and $K$ is a number field of prime degree such that $|Y_0(n)(\Q)|<|Y_0(n)(K)|<+\infty$ are listed in the table below.

\begin{center}
\begin{tabular}{|c|c|c|}
\hline
$n$ & $K$ & $|Y_0(n)(K)|$ \\
\hline
$14$ & $\Q (\sqrt{-7})$ & $8$\\
\hline
$14$ & $\Q (\sqrt{-3})$ & $14$\\
\hline
$15$ & $\Q (\sqrt{5})$ &  $12$\\
\hline
$15$ & $\Q (i)$ &  $12$\\
\hline
$20$ & $\Q(i)$ &  $6$\\
\hline
$21$ & $\Q(\sqrt{-3})$ & $12$\\
\hline
$27$ & $\Q(\sqrt{-3})$ & $7$\\
\hline
$32$ & $\Q(i)$ & $4$\\
\hline
$49$ & $\Q(\sqrt{-7})$ & $2$\\
\hline

\end{tabular}

\medskip

Table 2.
\end{center}
\end{tm3}

\begin{proof}
As the general strategy of the proof is similar as the proof of Theorem \ref{t1}, we will leave some details out. We will not explicitly write the Weierstrass models of $X_0(n)$, but note that they can be found in \cite{yan}.

The proof will be different for elliptic curves with complex multiplication and those without. We first deal with those without complex multiplication using the same general strategy as in the proof of Theorem \ref{t1}.

We first deal with the cases when $X_0(n)$ does not have complex multiplication, i.e. $n\in \{11,14,15,17,19,20,21,24\}$. As in the proof of Theorem \ref{t1}, for each $X_0(n)$ we compute primes $l$ for which the Galois representation $\rho_l$ is not surjective. We list the obtained results in the table below.

\begin{center}
\begin{tabular}{|c|c|}
\hline
\text{Modular curve} & \text{primes $l$ for which $\rho_l$ is not surjective}\\
\hline
$X_0(11)$ & 5\\
\hline
$X_0(14)$ & 2, 3\\
\hline
$X_0(15)$ & 2\\
\hline
$X_0(17)$ & 2\\
\hline
$X_0(19)$ & 3 \\
\hline
$X_0(20)$ & 2, 3 \\
\hline
$X_0(21)$ & 2 \\
\hline
$X_0(24)$ & 2 \\
\hline
\end{tabular}

\medskip

Table 3.
\end{center}
We start by proving that there is no number field $K$ of prime degree such that
$|X_0(11)(\Q)|<|X_0(11)(K)|<+\infty$. We factor the $5$-division polynomial of $X_0(11)$ and obtain that $x^2 + x - 29/5$ is the only factor of prime degree, but $X_0(11)(\Q(\sqrt{5}))= X_0(11)(\Q)\simeq \Z / 5\Z$.

The $25$-division polynomial of $X_0(11)$ has no factors of prime degree, except the ones corresponding to the $5$-division polynomial. Thus, we have dealt with $X_0(11)$.

By examining $X_0(14)$ over the number fields generated by the factors of prime degree of the $4$-division and $9$-division polynomials, we obtain
$$X_0(14)(\Q(\sqrt{-7}))_{tors} \simeq \Z / 2\Z \oplus \Z / 6\Z \text{ and } X_0(14)(\Q(\sqrt{-3}))_{tors} \simeq \Z / 3\Z \oplus \Z / 6\Z$$
and $X_0(14)(K)_{tors}=X_0(14)(\Q)_{tors}$ for all other fields $K$ of prime degree. We compute that the rank of $X_0(14)$ over both $\Q(\sqrt{-3})$ and $\Q(\sqrt{-7})$ is 0. Thus $Y_0(14)(\Q(\sqrt{-7}))$ has 8 points ($|X_0(14)(\Q(\sqrt{-7}))|$ minus the 4 rational cusps) and $Y_0(14)(\Q(\sqrt{-3}))$ has 14 points.

By examining $X_0(15)$ over the number fields generated by the factors of prime degree of the $8$-division polynomial, we obtain
$$X_0(15)(\Q(i))_{tors} \simeq \Z / 4\Z \oplus \Z / 4\Z \text{ and } X_0(15)(\Q(\sqrt{5}))_{tors} \simeq \Z / 2\Z \oplus \Z / 8\Z$$
and $X_0(15)(K)_{tors}=X_0(15)(\Q)_{tors}$ for all other fields $K$ of prime degree. We compute that the rank of $X_0(15)$ over both $\Q(i)$ and $\Q(\sqrt{5})$ is 0. Thus, by removing the 4 rational cusps of $X_0(15)$, we obtain that $Y_0(15)(\Q(i))$ and $Y_0(15)(\Q(\sqrt{5}))$ have 12 points.

We factor the 8-division polynomial of $X_0(17)$ and obtain that the only number field $K$ of prime degree such that $X_0(17)$ has torsion points defined over $K$, but not over $\Q$, is $K=\Q(i)$. But we compute that the rank of $X_0(17)(\Q(i)$) is 1, and hence there are infinitely many elliptic curves with a cyclic 17-isogeny over $\Q(i)$.

We factor the 9-division polynomial of $X_0(19)$ and obtain that the only number field $K$ of prime degree such that $X_0(19)$ has torsion points defined over $K$, but not over $\Q$, is $K=\Q(\sqrt{-3})$. But, as in the previous case, the rank of $X_0(19)(\Q(\sqrt{-3}))$ is 1, and hence there are infinitely many elliptic curves with a cyclic 19-isogeny over $\Q(\sqrt{-3})$.

By examining $X_0(20)$ over the number fields generated by the factors of prime degree of the $4$-division and $9$-division polynomials, we obtain
$$X_0(20)(\Q(i))_{tors} \simeq \Z / 2\Z \oplus \Z / 6\Z$$
and $X_0(20)(K)_{tors}=X_0(20)(\Q)_{tors}$ for all other fields $K$ of prime degree. We compute that the rank of $X_0(20)$ over $\Q(i)$ is 0, and hence $Y_0(20)(\Q(i))$ has 6 points (we subtract the 6 rational cusps of $X_0(20)$).

In a similar manner we obtain that the only field of prime degree $K$ such that $X_0(21)$ has torsion points defined over $K$, but not over $\Q$, is $K=\Q(\sqrt{-3})$, and that the rank of $X_0(21)(K)$ is 0. Hence $Y_0(21)(\Q(\sqrt{-3}))$ has 12 points (we subtract the 4 rational cusps of $X_0(21)$).

By examining $X_0(24)$ over the number fields generated by the factors of prime degree of the $8$-division, we obtain that $X_0(24)(K)_{tors}=X_0(24)(\Q)_{tors}$ for all number fields of prime degree.

For the remaining cases $n\in S$, $X_0(n)$ is a curve with complex multiplication, so we cannot apply the same methods as before. Instead we use following theorem which give a good description of the behavior of the torsion of an elliptic curve with complex multiplication defined over $\Q$ upon extensions.
\newtheorem{tmd}[tm]{Theorem}
\begin{tmd}\cite[Theorem 2]{dgu}
Let $E$ be an elliptic curve defined over $\Q$ with CM by an order of $K=\Q(\sqrt{-D})$ and $p$ an odd prime not dividing $D$. Let $F$ be a Galois number field not containing $K$. Then $E(F)[p]$ is trivial.
\label{tmd}
\end{tmd}

To find all the number fields $K$ of prime degree such that $|X_0(32)(\Q)|<|X_0(32)(K)|<+\infty$, by Theorem \ref{tmd}, we just need to check the 8-division polynomial of $X_0(32)$ and determine $X_0(32)(\Q(i))$. We obtain that the only field $K$ with the desired property is $K=\Q(i)$, and
$$X_0(32)(\Q(i))_{tors} \simeq \Z / 2\Z \oplus \Z / 4\Z.$$
The rank of $X_0(32)(\Q(i))$ is 0, and hence we obtain that $Y_0(32)(\Q(i))$ has 4 points.

Similarly, for $X_0(36)$,  by Theorem \ref{tmd}, we need to check the 9-division polynomial and $X_0(36)(\Q(\sqrt{-3}))$. We obtain
$$X_0(36)(\Q(\sqrt{-3}))_{tors} \simeq \Z / 2\Z \oplus \Z / 6\Z$$
and $X_0(36)(K)_{tors}=X_0(36)(\Q)_{tors}$ for all other fields $K$ of prime degree. The rank of $X_0(36)(\Q(\sqrt{-3}))$ is 0, and hence $Y_0(36)(\Q(\sqrt{-3}))$ has 6 points.

Finally, for $X_0(49)$, by Theorem \ref{tmd}, we need to check the 4-division and 7-division polynomials and $X_0(49)(\Q(\sqrt{-7}))$.  We obtain
$$X_0(49)(\Q(\sqrt{-7}))_{tors} \simeq \Z / 2\Z \oplus \Z / 2\Z$$
and $X_0(49)(K)_{tors}=X_0(49)(\Q)_{tors}$ for all other fields $K$ of prime degree. The rank of $X_0(49)(\Q(\sqrt{-7}))$ is 0, and hence $Y_0(49)(\Q(\sqrt{-7}))$ has 2 points.
\end{proof}

\textbf{Acknowledgements}.

The author is grateful to Andrej Dujella, Ivica Gusi\'c and the anonymous referee for helpful comments and suggestions.

\small{DEPARTMENT OF MATHEMATICS,\\ UNIVERSITY OF ZAGREB,\\ BIJENI\v CKA CESTA 30,\\ 10000 ZAGREB, CROATIA}\\
\emph{E-mail:} fnajman@math.hr

\end{document}